\documentclass[10pt,a4paper]{article}
\usepackage[utf8]{inputenc}
\usepackage[dvipsnames]{xcolor}
\usepackage{graphicx}
\usepackage{hyperref}
\usepackage{mathtools}
\usepackage{amsmath}
\usepackage{amsthm}
\usepackage{commath}
\usepackage{mathabx}
\usepackage{tikz}
\usepackage[labelformat = simple]{subcaption}
\usepackage{hyperref}
\usepackage{xspace}
\usepackage{comment}
\usepackage[hmargin=2.5cm,vmargin=3.8cm]{geometry}
\usepackage[textsize=footnotesize,color=green!40]{todonotes}
\newtheorem{theorem}{Theorem}[section]

\newtheorem{corollary}[theorem]{Corollary}

\newtheorem{lemma}[theorem]{Lemma}
\newtheorem{proposition}[theorem]{Proposition}

\newtheorem{definition}[theorem]{Definition}

\DeclarePairedDelimiter\floor{\lfloor}{\rfloor}
\DeclareMathOperator{\fes}{fes}
\DeclareMathOperator{\MEG}{meg}
\DeclareMathOperator{\MLN}{mln}

\newcommand{\decisionpb}[4]{
\begin{center}
        \noindent\framebox{\begin{minipage}{#4\textwidth}
                #1\\
                Instance: #2\\ 
                Question: #3
        \end{minipage}}%\vspace{0.5\baselineskip}
\end{center}
}

\title{Monitoring edge-geodetic sets in graphs}

\author{
Subhadeep R. Dev\footnote{Indian Statistical Institute, Kolkata, India}
\and Sanjana Dey\footnote{National University of Singapore, Singapore}~\footnote{\noindent UMONS -- Université de Mons, Mons, Belgium}
\and Florent Foucaud\footnote{Université Clermont Auvergne, CNRS, Clermont Auvergne INP, Mines Saint-Etienne, LIMOS, 63000 Clermont-Ferrand, France.} \and Krishna Narayanan\footnote{\noindent Department of Mathematics, Simon Fraser University, Burnaby, Canada}~\footnote{\noindent PSG College of Technology, Coimbatore, India}
\and Lekshmi Ramasubramony Sulochana \footnotemark[6]}

\begin{document}
\maketitle

\begin{abstract}
We introduce a new graph-theoretic concept in the area of network monitoring. In this area, one wishes to monitor the vertices and/or the edges of a network (viewed as a graph) in order to detect and prevent failures. Inspired by two notions studied in the literature (edge-geodetic sets and distance-edge-monitoring sets), we define the notion of a monitoring edge-geodetic set (MEG-set for short) of a graph $G$ as an edge-geodetic set $S\subseteq V(G)$ of $G$ (that is, every edge of $G$ lies on some shortest path between two vertices of $S$) with the additional property that for every edge $e$ of $G$, there is a vertex pair $x, y$ of $S$ such that $e$ lies on \emph{all} shortest paths between $x$ and $y$. The motivation is that, if some edge $e$ is removed from the network (for example if it ceases to function), the monitoring probes $x$ and $y$ will detect the failure since the distance between them will increase.

We explore the notion of MEG-sets by deriving the minimum size of a MEG-set for some basic graph classes (trees, cycles, unicyclic graphs, complete graphs, grids, hypercubes, corona products...) and we prove an upper bound using the feedback edge set of the graph.

We also show that determining the smallest size of an MEG-set of a graph is NP-hard, even for graphs of maximum degree at most~9.
\end{abstract}

\section{Introduction}
We introduce a new graph-theoretic concept, that is motivated by the problem of network monitoring, called \emph{monitoring edge-geodetic sets}. In the area of network monitoring, one wishes to detect or repair faults in a network; in many applications, the monitoring process involves distance probes~\cite{r2,BEEHHMR06,BR06,r1j}. Our networks are modeled by finite, undirected simple connected graphs, whose vertices represent systems and whose edges represent the connections between them. We wish to monitor a network such that when a connection (an edge) fails, we can detect the said failure by means of certain probes. To do this, we select a small subset of vertices (representing the probes) of the network such that all connections are covered by the shortest paths between pairs of vertices in the network. Moreover, any two probes are able to detect the current distance that separates them. The goal is that, when an edge of the network fails, some pair of probes detects a change in their distance value, and therefore the failure can be detected. Our inspiration comes from two areas: the concept of \emph{geodetic sets} in graphs and its variants~\cite{r4}, and the concept of \emph{distance edge-monitoring sets}~\cite{r1,r1j}.

We now proceed with some necessary definitions. A \emph{geodesic} is a shortest path between two vertices $u,v$ of a graph $G$~\cite{b2}. The \emph{length} of a path is the number of its edges, and the length of a geodesic between two vertices $u,v$ in $G$ is the \emph{distance} $d_G(u,v)$ between them. For an edge $e$ of $G$, we denote by $G-e$ the graph obtained by deleting $e$ from $G$. An edge $e$ in a graph $G$ is a \emph{bridge} if $G-e$ has more connected components than $G$. 
The \emph{open neighborhood} of a vertex $v \in V(G)$ is $N_G(v) = \{u \in V \, | \, uv \in E(G)\}$ and its \emph{closed neighborhood} is the set $N_G[v] = N_G(v) \cup \{v\}$.

\paragraph{Monitoring edge-geodetic sets.}
We now formally define our main concept. 

\begin{definition}
Two vertices $x,y$ \emph{monitor} an edge $e$ in graph $G$ if $e$ belongs to all shortest paths between $x$ and $y$. A set $S$ of vertices of $G$ is called a \emph{monitoring edge-geodetic set} of $G$ (\emph{MEG-set} for short) if, for every edge $e$ of $G$, there is a pair $x,y$ of vertices of $S$ that monitors $e$.
\end{definition}

We denote by $\MEG(G)$ the size of a smallest MEG-set of $G$. We note that $V(G)$ is always an MEG-set of $G$, thus $\MEG(G)$ is always well-defined.

\paragraph{Related notions.}
A set $S$ of vertices of a graph $G$ is a \emph{geodetic set} if every vertex of $G$ lies on some shortest path between two vertices of $S$~\cite{r4}. An \emph{edge-geodetic set} of $G$ is a set $S \subseteq V(G)$ such that every edge of $G$ is contained in a geodesic joining some pair of vertices in $S$~\cite{r5}. A \emph{strong edge-geodetic set} of $G$ is a set $S$ of vertices of $G$ such that for each pair $u,v$ of vertices of $S$, one can select a shortest $u-v$ path, in a way that the union of all these $\binom{|S|}{2}$ paths contains $E(G)$~\cite{SEGS}. It follows from these definitions that any strong edge-geodetic set is an edge-geodetic set, and any edge-geodetic set is a geodetic set (if the graph has no isolated vertices). In fact, every MEG-set is a strong edge-geodetic set. Indeed, given an MEG-set $S$, one can choose any shortest path between each pair of vertices of $S$, and the set of these paths covers $E(G)$. Indeed, every edge of $G$ is contained in \emph{all} shortest paths between some pair of $S$. Hence, MEG-sets can be seen as an especially strong form of strong edge-geodetic sets.

A set $S$ of vertices of a graph $G$ is a \emph{distance-edge monitoring} set if, for every edge $e$, there is a vertex $x$ of $S$ and a vertex $y$ of $G$ such that $e$ lies on all shortest paths between $x$ and $y$~\cite{r1,r1j}. Thus, it follows immediately that any MEG-set of a graph $G$ is also a distance-edge monitoring set of $G$.

\paragraph{Our results.} 
 We start by presenting some basic lemmas about the concept of MEG-sets in Section~\ref{sec:prelim}, that are helpful for understanding this concept. We then study in Section~\ref{sec:classes} the optimal value of $\MEG(G)$ when $G$ is a tree, cycle, unicyclic graph, complete (multipartite) graph, hypercubes, grids and corona products. In Section~\ref{sec:fes}, we show that $\MEG(G)$ is bounded above by a linear function of the \emph{feedback edge set number} of $G$ (the smallest number of edges of $G$ needed to cover all cycles of $G$, also called \emph{cyclomatic number}) and the number of leaves of $G$. This implies that $\MEG(G)$ is bounded above by a function of the \emph{max leaf number} of $G$ (the maximum number of leaves in a spanning tree of $G$). These two parameters are popular in structural graph theory and in the design of algorithms. We refer to Figure~\ref{fig:diagram} for the relations between parameter $\MEG$ and other graph parameters. We show  that determining $\MEG(G)$ is an NP-complete problem, even in graphs of maximum degree at most~9, in Section~\ref{sec:NPc}. Finally, we conclude in Section~\ref{sec:conclu}.

\begin{figure}[ht!]
\centering
\tikzstyle{mybox}=[line width=0.5mm,rectangle, minimum height=.8cm,fill=white!70,rounded corners=1mm,draw]
\tikzstyle{myedge}=[line width=0.5mm,<-]
\newcommand{\tworows}[2]{\begin{tabular}{c}{#1}\\[-1mm]{#2}\end{tabular}}
\scalebox{0.8}{\begin{tikzpicture}[node distance=10mm]
 		\node[mybox] (vc)  {Vertex Cover Number};
 		 \node[mybox] (mln) [above right = of vc,xshift=15mm,yshift=-2mm] {Max Leaf Number};
  		\node[mybox] (fes) [right = of vc] {\tworows{Feedback Edge Set}{Number}}  edge[myedge] (mln);
  		 \node[mybox] (meg) [right = of fes,fill=black!15] {$\MEG$} edge[myedge] (mln);
 		\node[mybox] (fvs) [below = of vc, yshift=5mm] {\tworows{Feedback Vertex Set}{Number}} edge[myedge] (vc) edge[myedge] (fes);
% 		\node[mybox] (td) [left = of fvs] {Treedepth} edge[myedge] (vc);
% 		\node[mybox] (bw) [left = of td] {Bandwidth};
% 		\node[mybox] (pw) [below = of td] {Pathwidth} edge[myedge] (td) edge[myedge] (bw);
 		\node[mybox] (dem) [right = of fvs] {\tworows{Distance Edge-Monitoring}{Number}} edge[myedge] (fes) edge[myedge] (vc) edge[myedge] (meg);
 		 \node[mybox] (segs) [right = of dem] {\tworows{Strong Edge-Geodetic}{Set Number}} edge[myedge] (meg);
 		  \node[mybox] (egs) [below = of segs, yshift=5mm] {\tworows{Edge-Geodetic}{Set Number}} edge[myedge] (segs);
 		  \node[mybox] (gs) [below = of egs, yshift=5mm] {\tworows{Geodetic Set}{Number}} edge[myedge] (egs);
 		\node[mybox] (arb) [below = of fvs,xshift=20mm, yshift=5mm] {Arboricity} edge[myedge] (fvs) edge[myedge] (dem);%edge[myedge] (pw) 
% 		\node[mybox] (c) [below = of arb] {Clique Number} edge[myedge] (arb);

\end{tikzpicture}}
\caption{Relations between the parameter $\MEG$ and other structural parameters in graphs (with no isolated vertices). For the relationships of distance edge-monitoring sets, see~\cite{r1,r1j}. Arcs between parameters indicate that the value of the bottom parameter is upper-bounded by a function of the top parameter.}
\label{fig:diagram}
\end{figure}
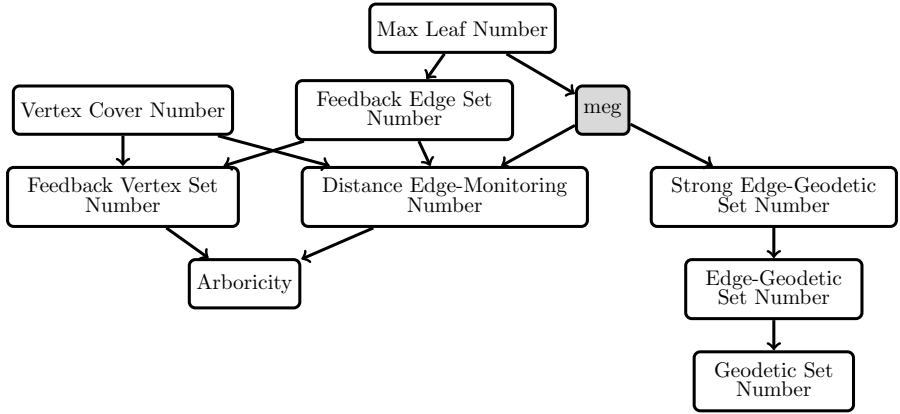

\paragraph{Further related work on MEG-sets.} This paper is the full version of a paper presented at the CALDAM 2023 conference~\cite{CALDAM}, where the notion of MEG-sets was presented for the first time. It contains the full proofs of the results in the conference version (with a corrected version of Theorem~\ref{uct}), as well as the new result on corona products, and the new NP-completeness proof. In the meantime, Haslegrave~\cite{H23} proved that the MEG-set problem is NP-complete on general graphs, and also studied MEG-sets for other graph products.

\section{Preliminary lemmas}\label{sec:prelim}

We now give some useful lemmas about the basic properties of MEG-sets.

A vertex is \emph{simplicial} if its neighborhood forms a clique.

\begin{lemma}\label{simplicial}\label{leaf-node}
In a graph $G$ with at least one edge, any simplicial vertex (in particular, any leaf) belongs to any edge-geodetic set and thus, to any MEG-set of $G$.
\end{lemma}
\begin{proof} Let us consider by contradiction an MEG-set of $G$ that does not contain said simplicial vertex $v$. Any shortest path passing through its neighbors will not pass through $v$, because all the neighbors are adjacent, hence leaving the edges incident to $v$ uncovered, a contradiction.
\end{proof}

Two distinct vertices $u$ and $v$ of a graph G are \emph{open twins} if $N(u) = N(v)$ and \emph{closed twins} if $N[u] = N[v]$. Further, $u$ and $v$ are \emph{twins} in $G$ if they are open twins or closed twins in $G$.

\begin{lemma}\label{twinv}
If two vertices are twins of degree at least~1 in a graph $G$, then they must belong to any MEG-set of $G$.
\end{lemma}
\begin{proof}
For any pair $u,v$ of open twins in $G$, for any shortest path passing through $u$, there is another one passing through $v$. Thus, if $u,v$ were not part of the MEG-set, then the edges incident to $u$ and $v$ would remain unmonitored, a contradiction.

If $u,v$ are closed twins, if some shortest path contains the edge $uv$, then it must be of length~1 and consist of the edge $uv$ itself (otherwise there would be a shortcut). Thus, to monitor $uv$, both $u,v$ must belong to any MEG-set.
%$u$ and $v$ are a part of their respective neighbourhoods. This allows for an edge $uv$, owing to which all other vertices in their respective neighnourhoods need to be a part of the optimal \emph{MEG-set} of the graph to ensure that no edges are unmonitored. 
\end{proof}

The next two lemmas concern cut-vertices and subgraphs, and will be useful in some of our proofs.

\begin{lemma}\label{cutv}
Let $G$ be a graph with a cut-vertex $v$ and $C_1,C_2,\ldots,C_k$ be the $k$ components obtained when removing $v$ from $G$. If $S_1, S_2,\ldots, S_k$ are MEG-sets of the induced subgraphs $G[C_1 \cup \{v\}], G[C_2 \cup \{v\}],\ldots,G[C_k \cup \{v\}]$, then $S = (S_1\cup S_2,\ldots,\cup S_k)\setminus\{v\}$ is an MEG-set of $G$.
\end{lemma}
\begin{proof}
Consider any edge $e$ of $G$, say in $C_1$. Then, there are two vertices $x,y$ of $S_1$ such that $e$ belongs to all shortest paths between $x$ and $y$ in $G_1=G[C_1\cup\{v\}]$. Assume first that $v\notin\{x,y\}$. All shortest paths between $x$ and $y$ in $G$ also exist in $G_1$. Thus, $e$ is monitored by $\{x,y\}\subseteq S$ in $G$. Assume next that $v \in\{x,y\}$: without loss of generality, $v = x$. At least one edge exists in $G[C_2 \cup \{v\}]$, which implies that $S_2\setminus\{v\}$ is nonempty, say, it contains $z$. Then, $e$ is monitored by $y$ and $z$, since $z\in S$. Thus, $S$ monitors all edges of $G$, as claimed.
\end{proof}

\begin{lemma}\label{subgr}
Let $G$ be a graph and $H$ an induced subgraph of $G$ such that for all vertex pairs $\{x, y\}$ in  $H$, no shortest path between them uses edges in $G-H$. Then, for any set $S$ of vertices of $G$ containing an MEG-set of $H$, the edges of $H$ are monitored by $S$ in $G$. 
\end{lemma}
\begin{proof}
Consider a subset $S \subseteq V(G)$ containing an MEG-set $S'$ of $H$. Let $e$ be an edge in $H$ that lies on all shortest paths between some pair $\{x,y\}$ of vertices of $S'$. By our hypothesis, no shortest path between $x$ and $y$ in $G$ uses any edges of $G-H$. Thus, the shortest paths between $x$ and $y$ in $H$ are the same as in $G$, and therefore in $G$, $e$ is also monitored by $\{x,y\}\subseteq S$.
\end{proof}

\section{Basic graph classes and bounds}\label{sec:classes}
In this section, we study MEG-sets for some standard graph classes.

\subsection{Trees}

\begin{theorem}\label{tt}\label{thm:trees}
For any tree $T$ with at least one edge, the only optimal MEG-set of $T$ consists of the set of leaves of $T$.
\end{theorem}
\begin{proof} The fact that all leaves must be part of any MEG-set follows from Lemma~\ref{leaf-node}, as they are simplicial. For the other side, let $L$ be the set of leaves of $T$. Let $e = xy$ be an edge of $T$ and consider two leaves of $T$, $l_x$ and $l_y$, such that $l_x$ is closer to $x$ than to $y$ and that $l_y$ is closer to $y$ than to $x$. We note that $e$ belongs to the unique (shortest) path between $l_x$ and $l_y$, thus $e$ is monitored by $L$. Hence, $L$ is an MEG-set of $T$. 
\end{proof}

\begin{corollary}
For any path graph $P_n$, where $n\geq2$, we have $\MEG(P_n) = 2$.
\end{corollary}
This provides a lower bound which is tight for path graphs, which have order $n$ and exactly 2 leaves.

\begin{corollary}
For any tree $T$ of order $n\geq 3$, we have $2 \leq\MEG(T)\leq n-1$.
\end{corollary}

The upper bound is tight for star graphs, which have order $n$ and $n-1$ leaves.

\subsection{Cycle graphs}

\begin{theorem}\label{cyclet}
Given an $n$-cycle graph $C_n$, for $n = 3$ and $n \geq 5$, $\MEG(C_n) = 3$. Moreover, $\MEG(C_4) = 4$.
\end{theorem}
\begin{proof}
Let us first prove that we need at least three vertices to monitor any cycle. By contradiction, let us assume that two vertices suffice. For any arbitrary vertex pair in the cycle graph, there are two paths joining them, but there is either one single shortest path or two equidistant shortest paths between them. Thus, the edges on at least one of the two paths between the pair will not be monitored by it. Hence, we need at least three vertices in any MEG-set of $C_n$ ($n\geq 3$).

We now prove the upper bound. Let $n\geq 5$ or $n = 3$, with the vertices of $C_n$ from $v_0$ to $v_{n-1}$. Consider the set $S =\{v_0, v_{\floor{\frac{n}{3}}}$, $v_{\floor{\frac{2n}{3}}}\}$. We show that $S$ is an MEG-set of $C_n$.
 
Consider any edge of $C_n$ between a vertex pair $v_x$ and $v_y$ in $S$, then we note that it lies on every (unique) shortest path between these vertices, which has length at least~1 for $n\leq 5$ and at least~2 otherwise, and at most $\lfloor \frac{n-1}{2}\rfloor<\frac{n}{2}$. 
%$\frac{n}{2}-1$ when $n$ is even, and $\lfloor \frac{n-1}{2}\rfloor$ when $n$ is odd. 
Moreover, every edge of $C_n$ lies on such a path. Thus, $\MEG(C_n) = 3$ when $n \geq 5$ or $n=3$.

In the case of $C_4$, the above construction does not work. Consider a set of three vertices, say $v_0$, $v_1$, $v_2$ without loss of generality due to the symmetries of $C_4$. Notice that the edge $v_0v_3$ is unmonitored by this set. Thus, we have $\MEG(C_4)=4$.
\end{proof}

\subsection{Unicyclic graphs}

A \emph{unicyclic graph} is a connected graph containing exactly one cycle~\cite{b1}. We now determine the optimal size of an MEG-set of such graphs. Note that in the short version of this paper, published in the proceedings of CALDAM~\cite{CALDAM}, the statement was slightly mistaken. Here, we correct it.

\begin{theorem} \label{uct}
Let $G$ be a unicyclic graph where the only cycle $C$ has length $k$ and whose set of pendant vertices is $L(G)$, $|L(G)| = l$. Let $V_{c}^{+}$ be the set of vertices of $C$ with degree at least~3. Let $p(G)=1$  %if there exists a path of length at least $\lfloor\frac{k}{2}\rfloor$ between two vertices of $V_{c}^{+}$ that are consecutive on the cycle,
if $G[V(C)\setminus V_{c}^{+}]$ contains a path with at least $\lfloor\frac{k-1}{2}\rfloor$ vertices, 
and $p(G)=0$ otherwise.
Then, if $k\in \{3,4\}$,
\begin{equation*}
\MEG(G) = l + k - |V_c^{+}|.
\end{equation*}

Otherwise ($k \geq 5$), then
\begin{equation*}
\MEG(G) = 
\begin{cases}
3, \text{ if $|V_c^{+}|=0$;}\\
l + 2, \text{ if $|V_c^{+}|=1$;}\\
l + 2, \text{if $|V_c^{+}|=2$, $k$ is even, and the vertices in $V_c^{+}$ are either adjacent or opposite on $C$;}\\
%l + 2, \text{if $|V_c^+|=2$, $k$ is even and vertices in $V_c^+$ are opposite on $C$;}\\
l + p(G), \text{ in all other cases.}
%l + 1, \text{for $|V_c^{+}|>1$ and $p=1$}
\end{cases}    
\end{equation*}
\end{theorem}
\begin{proof}
Let $G$ be a unicyclic graph where the only cycle $C$ has length $k$ and whose set of pendant vertices is $L(G)$. By Lemma \ref{leaf-node}, all leaves are part of any MEG-set of $G$. This implies that $\MEG(G)$ is at least $l$. If $|V_C^+|=0$ (i.e. $l = 0$), we are done by Theorem~\ref{cyclet}, so let us assume $|V_C^+|>0$ and thus, $l>0$.

Similarly as in the proof of Lemma~\ref{cutv}, for every vertex $v$ of $V_C^+$, we know that at least one leaf will exist in the tree component $T_v$ formed if we remove the neighbors of $v$ in $C$ from $G$. Informally speaking, towards the rest of the graph, this leaf simulates the fact that $v$ is in the solution set.

If $k \in \{3,4\}$, we consider $S=L(G)$ and we add to $S$ all vertices of $C$ that are of degree~2 in $G$. One can easily check that this is an MEG-set. Moreover, one can see that adding these degree~2 vertices is necessary by using similar arguments as in the proof of Theorem~\ref{cyclet} on cycles.

Next, we assume that $k\geq 5$. Let $v_0,\ldots,v_{k-1}$ be the vertices of $C$. We have $p(G)=1$ if there is an edge of the cycle that is not on a unique shortest path between two vertices of $V_c^{+}$. In most cases, these edges would form a path in $G$, and require an additional vertex to monitor them. In some cases, these edges include all edges of $C$: when $k$ is even, $|V_c^{+}|=2$ and the two vertices of $V_c^{+}$ are opposite on the cycle. Then, we require two additional vertices to monitor them. In the special case where $k$ is even, $|V_c^{+}|=2$ and the two vertices of $V_c^{+}$ are adjacent, we also require two additional vertices. Let us now analyze these situations in detail.

When $|V_c^{+}| = 1$, without loss of generality, consider the vertex in $V_c^{+}$ to be $v_0$. Then, the vertices \(\{v_{\lfloor\frac{k}{3}\rfloor},v_{\lfloor\frac{2k}{3}\rfloor}\}\) on the cycle together with the pendant vertices are sufficient to monitor the graph, in the same way as in Theorem~\ref{cyclet}, showing $\MEG(G)\leq l+2$. Moreover, by the same arguments as in the proof of Theorem~\ref{cyclet}, one can see that if at most one vertex on $C$ is chosen in the MEG-set, some edge of $C$ will not be monitored, proving the lower bound $\MEG(G)\geq l+2$.

Consider the case where $k$ is even and $|V_c^{+}|=2$ such that the vertices in $V_c^{+}$ are adjacent (then, $p(G)=1$). Then, $G[V(C)\setminus V_{c}^{+}]$ consists of a single path $P$ with $k-2$ vertices, and the edges of $P$ (as well as the two edges joining $P$ to $V_{c}^{+}$) are not monitored by the set of leaves of $G$. We need to pick at least one vertex of $P$ to monitor these edges. If we pick only one vertex of $P$, then at least one of the edges of $P$ is not monitored (for example, the one that is, in the subgraph of $G$ induced by $P$, farthest from the picked vertex). Thus, we need to pick at least two additional vertices and $\MEG(G)\geq l+2$. Picking the two middle vertices of $P$, for example, is sufficient and yields an MEG-set of the desired size, $l+2$.

Similarly, if $k$ is even and $|V_c^{+}|=2$ such that the vertices in $V_c^{+}$ are opposite on $C$, there exist two paths of equal length between the two vertices of $V_c^{+}$, no edge of $C$ is monitored by the leaves of $G$, and $G[V(C)\setminus V_{c}^{+}]$ consists of two paths $P_1$ and $P_2$, each with $k/2-1$ vertices. We have $\MEG(G)\geq l+2$ because we need to pick at least one vertex of each of $P_1$ and $P_2$. We can in fact pick any vertex of $P_1$ and any vertex of $P_2$, and obtain an MEG-set of size $l+2$, as desired.

Let us next assume we are in neither of the previous cases.

If $|V_c^{+}|> 1$ and $p(G)=0$, the $l$ pendant vertices are sufficient to monitor $G$. Indeed, consider an edge $e$. If $e$ is not on $C$, let $v$ be the vertex of $V_C^+$ closest to $e$, and let $w\neq v$ be the vertex of $V_C^+$ closest to $v$ (it exists because $|V_c^{+}|>1$). Consider a leaf $f$ of $G$ such that $e$ lies on some path from $v$ to $f$. Since $p(G)=0$, the path from $w$ to $f$ is a unique shortest path, and thus, $e$ is monitored by $f$ and some leaf whose closest vertex on $C$ is $w$.

If $e$ is an edge of $C$, $e$ lies on a path of length strictly less than $\frac{k}{2}$ between two vertices $v,w$ of $V_C^+$. Since $p(G)=0$, this path is a unique shortest path between $v$ and $w$, and $e$ is monitored by two leaves, each of which has $v$ and $w$ as its closest vertex of $C$, respectively.

Finally, assume that $|V_c^{+}|>1$ and $p(G)=1$. Now, there is only one problematic path $P$ of $G[V(C)\setminus V_{c}^{+}]$ with at least $\lfloor\frac{k}{2}\rfloor$ vertices (since we already dealt with case where $k$ is even with two vertices in $V_c^{+}$ that are opposite on $C$). As in the previous cases, we can see that the edges of $P$ are not monitored by the set of leaves of $G$, which implies that $\MEG(G)\geq l+1$. To show that $\MEG(G)\leq l+1$, we select as an MEG-set, the set of leaves together with the middle vertex of $P$ (if $P$ has an odd number of vertices) or one of the middle vertices of $P$ (if $P$ has an even number of vertices). One can see that this is an MEG-set by similar arguments as in the previous cases.
\end{proof}

\subsection{Complete graphs}

The following follows immediately from Lemma~\ref{simplicial}, since every vertex of a complete graph is simplicial.

\begin{theorem}
 For any $n \geq 2$, we have $\MEG(K_n) = n$.
\end{theorem}

\subsection{Complete multipartite graphs}

The complete $k$-partite graph $K_{p_1,p_2,\ldots,p_k}$ consists of $k$ disjoint sets of vertices of sizes $p_1,p_2,\ldots,p_k$, with an edge between any two vertices from distinct sets.

\begin{theorem}
We have $\MEG(K_{p_1,p_2,\ldots,p_k}) = \abs{V(K_{p_1,p_2,\ldots,p_k})}$, with the exceptional case of a bipartite graph $K_{1,p}$ with an independent set of size 1 (a star graph), for which $\MEG(K_{1,p})=p$.
\end{theorem}
\begin{proof}
In a complete $k$-partite graph, all vertices in a given partite set are twins. Therefore, by Lemma~\ref{twinv}, all vertices of any partite set of size at least~2 need to be a part of any MEG-set.

If we have several partite sets of size~1, then the vertices from these sets are closed twins, and again by Lemma~\ref{twinv} they all belong to any MEG-set.

Thus, we are done, unless there is a unique partite set of size~1, whose vertex we call $v$. If there are at least three partite sets, then note that $v$ is never part of a unique shortest path, and thus the edges incident with $v$ cannot be monitored if $v$ is not part of the MEG-set.

On the other hand, if the graph is bipartite, it is a star $K_{1,p}$. Here, we know by Theorem~\ref{thm:trees} that $\MEG(G)=p$, as claimed.
\end{proof}

\subsection{Hypercubes}

The \emph{hypercube of dimension $n$}, denoted by $Q_n$, is the undirected graph consisting of $k = 2^n$ vertices labelled from $0$ to $2^n - 1$ and such that there is an edge between any two vertices if and only if the binary representations of their labels differ by exactly one bit~\cite{r8}. The \emph{Hamming distance} $H(A,B)$ between two vertices $A,B$ of a hypercube is the number of bits where the two binary representations of its vertices differ. 

We next show that not only $C_4$ has the whole vertex set as its only MEG-set (Theorem~\ref{cyclet}), but that this also holds for all hypercubes.

\begin{theorem}\label{hyperc}
For a hypercube graph $Q_n$ with $n \geq 2$, we have $\MEG(Q_n) = 2^n$.
\end{theorem}
\begin{proof}
Assume by contradiction that there is an MEG-set $M$ of size at most $2^n - 1$. Let $v \in V(G)$ be a vertex that is not in $M$. It is known that for every vertex pair $\{v_x, v\}$ with $H(v_x , v)\leq n$, there are $H(v_x , v)$ vertex-disjoint paths of length $H(v_x , v)$ between them~\cite{r8}. Thus, there is no vertex pair in $M$ with a unique shortest path going through the edges incident with $v$, and $M$ is not an MEG-set, a contradiction.
\end{proof}

\subsection{Grid graphs}

The graph $G \bigboxvoid H$ is the Cartesian product of graphs $G$ and $H$ and with vertex set $V(G \bigboxvoid H) = V(G)$ x $V(H)$, and for which $\{(x,u),(y,v)\}$ is an edge if $x = y$ and
$\{u,v\}\in E(H)$ or $\{x,y\}\in E(G)$ and $u = v$. The \emph{grid graph} $G(m,n)$ is the Cartesian product $P_m \bigboxvoid P_n$ with vertex set $\{(i,j)~|~1\leq i\leq m, 1\leq j\leq n\}$.

\begin{theorem}
For any $m,n\geq 2$, we have $\MEG(G(m,n)) = 2(m+n-2)$.
\end{theorem}
\begin{proof}
We claim that the set $S=\{(i,j)\in V(G(m,n)), ~i\in\{1,m\}~\text{and}~1\leq j\leq n~\text{or}~j\in\{1,n\}~\text{and}~1\leq i\leq m\}$ of $2(m+n-2)$ vertices of $G(m,n)$ that form the boundary vertices of the grid, form the only optimal MEG-set.

For the necessity side, let us assume that some vertex $v=(i,j)$ of $S$ is not part of the MEG-set. If $v$ is a corner vertex (without loss of generality say $v=(1,1)$, the two edges incident with $v$ are not monitored, as for any shortest path going through them, there is another one going through vertex $(2,2)$. If $v$ is not a corner vertex (without loss of generality say $v=(1,j)$ with $2\leq j\leq n-1$), then the edge $e$ between $v=(1,j)$ and $(2,j)$ is not monitored, indeed for any shortest path containing $e$, there is another one avoiding it, either going through vertex $(2,j-1)$ or through $(2,j+1)$.

To see that $S$ is an MEG-set, first see that each boundary edge is monitored by its endpoints. Next, consider an edge $e$ that is not a boundary edge, without loss of generality, $e$ is between $(i,j)$ and $(i+1,j)$. Then, it is monitored by $(1,j)$ and $(m,j)$, whose unique shortest path goes through $e$.  
\end{proof}

\subsection{Corona products}
The corona product $G \odot H$ of two graphs $G$ and $H$ is defined as the graph obtained by taking one copy of $G$ and $|V(G)|$ copies of $H$ and joining the $i^{th}$ vertex of $G$ to every vertex in the $i^{th}$ copy of $H$.

Let graph $G$ be simple and connected and graph $H$ be simple. Let $\{g_i\}$ and $\{h_i\}$ denote the vertex set of $G$ and $H$ respectively. Let $|V(G)| = n_1$ and $|V(H)| = n_2$.

\begin{theorem}
Given a simple, connected graph $G$ and a simple graph $H$, we have $meg(G \odot H) = n_1n_2$.
\end{theorem}
\begin{proof}
By definition of $G \odot H$, every vertex $g_i$ is a cut-vertex in it. Therefore, the edge $e = (g_i, g_j)$ lies in the unique shortest path connecting the $i^{th}$ copy $H_i$ with the $j^{th}$ copy $H_j$ in $G \odot H$. i.e. every edge of $G$ is in the unique shortest path of some $H_i$ and $H_j$. Hence no vertex of $G$ is in $meg(G \odot H)$.

If some vertex $h$ of $H_i$ is not in the MEG-set, consider an edge of $H_i$ incident with $h$, if that edge is monitored, then there are two vertices other than $h$ monitoring it; these two vertices must be in $H_i$ and at distance 2 from each other, but then there are two edge-disjoint paths of length 2 joining them (one through $h$ and one through $g_i$), a contradiction.
\end{proof}

\section{Relation to feedback edge set number}\label{sec:fes}
A feedback edge set of a graph $G$ is a set of edges which when removed from $G$ leaves a forest. The smallest size of such a feedback edge set of $G$ is denoted by $\fes(G)$ and is sometimes called the \emph{cyclomatic number} of $G$.

We next introduce the following terminology from \cite{r11}. A vertex is a \emph{core vertex} if it has degree at least~3. A path with all internal vertices of degree~2 and whose end-vertices are core vertices is called a \emph{core path}. Do note that we allow the two end-vertices to be equal, but that every other vertex must be distinct. A core path that is a cycle (that is, both end-vertices are equal) is a \emph{core cycle}. For the sake of distinction, a core path that is not a core cycle is called a \emph{proper core path}. We say that a (non-empty) path from a core vertex $u$ to a leaf $v$ is a \textit{leg} of $u$ if all internal vertices of the path have degree~2 ($u$ is not considered to be a part of the leg). The \emph{base graph} of a graph $G$ is the graph of minimum degree~2 obtained from $G$ by iteratively removing vertices of degree~1. A \emph{hanging tree} is a connected subtree of $G$ which is the union of some legs removed from $G$ during the process of creating the base graph $G_b$ of $G$, together with the single vertex of the base graph to which the last removed vertices were adjacent. Thus, $G$ can be decomposed into its base graph and a set of maximal hanging trees. The root of such a maximal hanging tree $T$ is the vertex common to $T$ and $G$.

See Figure~\ref{fig:examplegraph} for a graph whose core vertices are in red. It has two hanging trees, three core cycles, four proper core paths of length~4, and six proper core paths of length~1.

\begin{figure}[ht!]
    \centering
    \subfloat{{\includegraphics[width=0.8\textwidth]{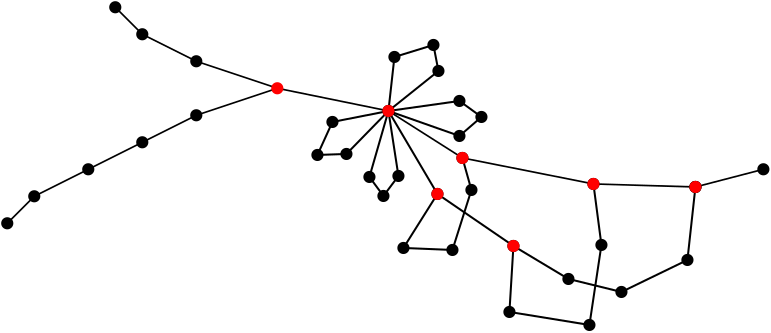} }}
    \caption{Example of a graph $G$ with its core vertices in red.}
    \label{fig:examplegraph}
\end{figure}

Based on the aforementioned, we have the following lemma.
\begin{lemma}[\cite{r11,r12}]\label{cores}
Let $G$ be a graph with $\fes(G) = k \geq 2$. The base graph of $G$ has at most $2k$ - $2$ core vertices, that are joined by at most $3k$ - $3$ edge-disjoint core paths. Equivalently, $G$ can be obtained from a multigraph $H$ of order at most $2k-2$ and size at most $3k-3$ by subdividing its edges an arbitrary number of times and iteratively adding degree~1 vertices.
\end{lemma}

\begin{lemma}\label{base}
Let $S$ be an MEG-set of the base graph $G_b$ of $G$ and $L(G)$ be the set of leaves in $G$. Then, $S \cup L(G)$ is an MEG-set of $G$.
\end{lemma}
\begin{proof}
Let $G_b$ be a base graph of $G$. Consider all vertices that are roots of maximal hanging trees on $G_b$. By Theorem~\ref{thm:trees}, the optimal MEG-set of each tree consists of all leaves. We repeatedly apply Lemma~\ref{cutv} to $G$, where for each application of Lemma~\ref{cutv}, the cut-vertex is the root of a hanging tree in consideration.
\end{proof}

Lemma~\ref{leaf-node}, Theorem~\ref{tt} and Lemma~\ref{base} together imply that if $\fes(G) = 0$, then $\MEG(G) \leq \fes(G)+|L(G)|$. Moreover, if $\fes(G)=1$, then $\MEG(G) \leq \fes(G) + |L(G)| + 3$, where $|L(G)|$ is the number of leaves of $G$. We next give a similar bound when $\fes(G)\geq 2$.

\begin{figure}[ht!]
    \centering
    \subfloat{{\includegraphics[width=0.35\textwidth]{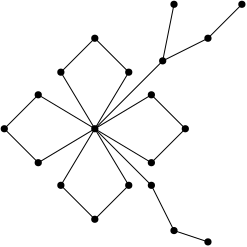} }}
    \qquad
    \subfloat{{\includegraphics[width=0.20\textwidth]{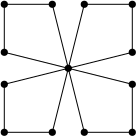} }}
    \caption{Example of a graph $G$ and its base graph $G_b$ with four core cycles of length~4 each.}
    \label{fig:corecycle}
\end{figure}

\begin{theorem}\label{thm:fes}
If $\fes(G) \geq 2$, then $\MEG(G) \leq 9\fes(G) + |L(G)| - 8$ where $|L(G)|$ is the number of leaves of $G$.
\end{theorem}
\begin{proof}
Let $k = \fes(G)$. We show how to construct a MEG-set $M$ of $G_b$ of order at most $9k-8$ and, by applying Lemma~\ref{base} to $G$, of order $9k-8 + |L(G)|$ for $G$. If an edge $e$ is part of a maximal hanging tree, then by Lemma~\ref{leaf-node} and Lemma~\ref{cutv}, it is monitored by the leaves of $G$ on the maximal hanging tree. $M$ is constructed as follows. 
\begin{itemize}
    \item We let all core vertices of $G_b$ be part of $M$.
    \item One or two internal vertices from each proper core path belongs to $M$, only if the length is at least~2, as explained below.
    \item Two or three internal vertices from each core cycle, as explained below.
\end{itemize}

Consider a proper core path $P$ of length at least~2, with core vertex endpoints $c$ and $c'$, and the median vertex $x_1$ in the case of an even-length path and $x_1$, $x_2$ in the case of an odd-length path, with $d$ edges (on $P$) between the endpoints and the respective medians in $P$. Then, we choose the single median vertex $x_1$ or the two median vertices $x_1$, $x_2$ from each of the core paths and add them to $M$.

For each core cycle $C$, in addition to the core vertex of that cycle, if $C$ has length~4, we add all three non-core vertices of $C$ to $M$. Otherwise, we add two non-core vertices of $C$ to $M$, so that now $C$ has three vertices in $M$. We do this so that these three vertices are as equidistant as possible on the cycle, to be part of $M$ (similar as in Theorem~\ref{cyclet}). % --- note that we may need three vertices if the cycle has length~4, so we conservatively add three vertices in all cases).
%For each core cycle, we conservatively assume the worst case that all core cycles are of length four and include three vertices apart from the core vertex to be part of $M$ (as in Theorem~\ref{cyclet}).

This finishes the description of the construction of $M$.

\medskip

We now show that $M$ monitors all edges of $G_b$. Let $e$ be any edge of $G$. If $e$ lies on a core cycle $C$, assume an origin core vertex of $v_0$. Then, based on Lemma~\ref{cutv} and Theorem~\ref{cyclet}, we deduce that in the worst case, $v_0$ and the two or three other vertices of $C$ in $M$ together suffice to monitor the edges.

If the edge $e$ lies on a proper core path $P$, then let $c$ and $c'$ be the core vertex endpoints of $P$. Let the median vertex of $P$ be $x_1$ in the case of an even-length path and $x_1$, $x_2$ the two medians in the case of an odd-length path. Assume that there are $d$ edges between the medians and closest endpoints of $P$. Without loss of generality, let us say that $e$ lies on the path $P$ such that its closest core vertex is $c$, and closest median $x_1$. 
Given that the distance between $c$ and $x_1$ is $d$ in $P$, the length of any other path between them must be at least $d+1$ (or $d+2$ if $P$ has odd length). Therefore, $c$ and $x_1$ monitor $e$, which lies on the unique shortest path of length~$d$ between them. %(We can similarly argue that if the closest core vertex to $e$ was $c'$ and the closest median vertex was $x_2$, then $c'$ and $x_2$ monitor $e$.) 
If $e$ lies in between the median vertices $x_1$ and $x_2$, then we know that those vertices would monitor $e$ because they are adjacent. %If the path was of even length, then depending on which of the core vertices $c$ and $c'$ was closest to $e$, the distance between the median and the core vertices would be $d$ in $P$ and the length of any other path between them at least $d+1$, ensuring that the median vertex $x_1$ would monitor the edge apart from the core vertices. %Therefore, based on this argument, in the worst case, apart from the core vertices, two median vertices $x_1$ and $x_2$ help monitor the edges on every proper core path $P_i$.
This justifies our construction of $M$, which is an MEG-set of $G_b$, as claimed.

\medskip

Let us now estimate the size of $M$. By Lemma~\ref{cores}, the number of core vertices of $G_b$ is at most $2k-2$, and there are at most $3k-3$ core paths.

If we have core cycles in our graph, then we must note that there can be at most $k$ such cycles in the graph. Indeed, if there were $k+1$ core cycles in the graph, since they are all edge-disjoint, we need at least $k+1$ edges to be removed from $G$ to obtain a forest, a contradiction to the fact that $\fes(G)=k$.

Let $n_c$ be the number of core cycles and $n_p$ be the number of proper core paths. Recall that after adding the core vertices to $M$, we added at most three vertices of each core cycle and two for each core path to $M$. Hence, we have $\abs{M} \leq 3n_c + 2n_p + 2k-2$. Since $n_c \leq k$ and $n_c + n_p \leq 3k-3$ by Lemma~\ref{cores}, we get $\abs{M} \leq 3k + 2(2k-3) + 2k-2 = 9k-8$.

Also note that this value is only reached if all core cycles are of length~4 and all core paths are of odd length.
\end{proof}

The \emph{max leaf number} of $G$, denoted $\MLN(G)$, is the maximum number of leaves in a spanning tree of $G$. It can be seen as a refinement of the feedback edge set number of $G$~\cite{E15}. We get the following corollary.

\begin{corollary}\label{cor:MLN}
For any graph $G$, we have $\MEG(G)= O(\MLN(G)^2)$, where $\MLN(G)$ is the max leaf number of $G$.
\end{corollary}
\begin{proof}
It is known that $\fes(G)= O(\MLN(G)^2)$~\cite{E15}, and clearly, $|L(G)|\leq\MLN(G)$, thus the bound follows from Theorem~\ref{thm:fes}.
\end{proof}

In the following, we show that Theorem~\ref{thm:fes} is best possible up to constant factors.

\begin{proposition}\label{cex}
For any integer $k \geq 2$, there exists a graph $G$ with $\fes(G) = k$ and $\MEG(G) = 3k + \abs{L(G)}$.
\end{proposition}
\begin{proof}
Consider $G$ and its base graph $G_b$ in Figure~\ref{fig:corecycle}. We know that the leaves must be part of any MEG-set by Lemma~\ref{leaf-node}. The MEG-set for $G_b$ consists of all the vertices in each of the core cycles (each a $C_4$) in $G_b$, except the common core vertex. It is easy to check that no smaller set can work. The size of the optimal MEG-set in this example is $3k + |L(G)|$ and therefore, this is an instance where this proposition holds.
\end{proof}

\section{NP-completeness for graphs of small maximum degree}\label{sec:NPc}

The \textsc{Monitoring Edge Geodetic Set} decision problem is defined as follows.

\decisionpb{\textsc{Monitoring Edge Geodetic Set}}{A graph $G=(V(G),E(G))$ and an integer $k$.}{Is there an \emph{MEG-set} $S \subseteq V(G)$ of $G$ of size at most $k$?}{.9}

In this section, we show that the problem is NP-hard by a reduction from \textsc{Vertex Cover}.

\decisionpb{\textsc{Vertex Cover}}{A graph $G=(V(G),E(G))$ and an integer $k$.}{Is there a vertex cover $C \subseteq V(G)$ of size at most $k$ such that every edge in $E(G)$ is incident with some vertex in $C$?}{0.9}

\paragraph{Overview of the reduction.} Consider a cubic planar graph. We subdivide each edge of this graph four times. Let this graph be $G$ with vertex set $V(G)$ and edge set $E(G)$. Also, let $|V(G)| = n$. Observe that the vertices of $G$ have degree 2 or 3 and its girth is at least 15. It is well know that \textsc{Vertex Cover} is NP-hard when the input graph is cubic, see for example~\cite{mohar2001face}. We show that it is also NP-hard when the input graph is restricted to resemble graphs such as $G$. (Note that such a result seems to be already known for \textsc{Independent Set}~\cite{alekseev1982effect}, which has the same complexity as \textsc{Vertex Cover}. We state the proof for \textsc{Vertex Cover} for completeness, as~\cite{alekseev1982effect} is in Russian.)

Then, given $G$, we construct a new graph $H_G$ of maximum degree 8 and show that $G$ has a vertex cover of size $k$ if and only if $H_G$ has an MEG-set of size $3n +k$. This completes our proof.

\begin{lemma}[\cite{alekseev1982effect}] \label{lem:vertex_cover}
\textsc{Vertex Cover} is NP-complete, even for input graphs obtained from a cubic graph by subdividing every edge four times.
%    $G'$ has a vertex cover of size $k$ if and only if $G$ has a vertex cover of size $2m' + k$.
\end{lemma}

\begin{proof}
Let $G'$ be the original (cubic planar) graph before subdivision and let $|E(G')| $ $ = m'$. To show that finding the minimum vertex cover in $G$ is NP-hard it suffices to show that $G'$ has a vertex cover of size $k$ if and only if $G$ has a vertex cover of size $2m' + k$.

\medskip

    \textbf{(The if part.)} Let $C$ be a vertex cover of $G$ of size $2m' + k$. We contruct a vertex cover $C'$ of $G'$ of size at most $k$ as follows. Consider all \emph{maximal paths} of $G$ of length 5 whose end points are degree 3 vertices and whose internal points are degree 2 vertices. Let $\pi_{uv}, u,v \in V(G')$, represent one such path. Observe that $\pi_{uv}$ which is in $G$, represents the edge $uv$ in $G'$. Since $C$ is a vertex cover, all edges in $\pi_{uv}$ are covered by some vertex in $C$. We add to $C'$ the vertex $u$ and/or $v$ if it was included in $C$. If both $u$ and $v$ were not in $C$ then we arbitrarily choose one of these two vertices and add it to $C'$.
    
    Since for each edge of $G'$ we have included one of its end points, $C'$ is a vertex cover of $G$. Also, $\pi_{uv}$ needs at least 3 vertices to cover all its edges. Therefore, if one of $u$ or $v$ is inserted into $C'$, two other vertices of $\pi_{uv}$, which are in $C$, are excluded. Again, if $u$ and $v$ both belong to $C$ then $\pi_{uv}$ needs at least 4 vertices to cover all its edges. Therefore, if both $u$ and $v$ are inserted into $C'$, two other vertices of $\pi_{uv}$, which are in $C$, are excluded. See Figure \ref{fig:maximal_path} for reference. Since for every such maximal path we excluded at least 2 vertices from $C$ from being inserted into $C'$, the size of $C'$ is at most $k$.

    \begin{figure}[!ht]
        \centering
        \includegraphics[width = .5 \textwidth]{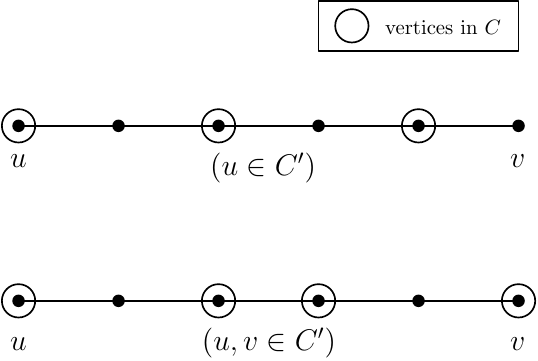}
        \caption{Choice of $C$ when $u \in C'$ and $u,v \in C'$.} \label{fig:maximal_path}
    \end{figure}

    \smallskip \noindent \textbf{(The only if part.)} Let $C'$ be a vertex cover of $G'$ of size $k$. We construct a vertex cover of $G$ of size at most $2m' + k$ as follows. As before, consider all maximal paths of $G$ and let $\pi_{uv}, u,v \in V(G')$, be one such path. Either $u$ or $v$  need to be in $C'$. We insert $u$ or $v$ or both into $C$ if they belong to $C'$ and are not already inserted. We select two additional vertices of $\pi_{uv}$ such that all its edges are covered and insert them into $C'$. As evident from Figure \ref{fig:maximal_path} this is always possible. $C$ is a vertex cover of $G$. Since for each maximal path $\pi_{uv}$ we insert two additional vertices into $C$, the size of $C$ is at most $2m' + k$.
\end{proof}

Now we show how to construct $H_G$ from $G$.

\paragraph{Construction.} For each degree 3 vertex $u \in V(G)$, we construct a \emph{vertex-gadget} $H_G^u$, with 16 vertices and 18 edges as shown in Figure~\ref{fig:vertex_gadget_3}. Each vertex $f_i^u, i \in \{1,2,3\},$ is associated with a unique side of the triangle $\triangle c_1^u c_2^u c_3^u$ as depicted in the figure. For example, $f_1^u$ is associated with the side $\overline{c_1^u d_1^u c_2^u}$. Each $f_i^u$ also has a unique \emph{associate vertex} $c_i^u$. Note that $c_i^u$ belongs to the side of $\triangle c_1^u c_2^u c_3^u$ associated with $f_i^u$ and its distance from $a^u$ is 2. For a degree 2 vertex $v \in V(G)$, we construct a \emph{vertex-gadget} $H_G^v$, with 14 vertices and 15 edges as shown in Figure~\ref{fig:vertex_gadget_2}. Again, each $f_i^v, i \in \{1,2\}$, vertex is associated with a unique side of the triangle $\triangle c_1^v c_2^v c_3^v$ depicted in the figure. In this case however, the assosiate vertex of $f_1^v$ is $c_1^v$ and that of $f_2^v$ is $c_3^v$. Again note that the associate vertex of $f_i^v$ belongs to the side of $\triangle c_1^v c_2^v c_3^v$ associated with $f_i^u$ and its distance from $a^v$ is 2. Also, observe that the missing elements in $H_G^v$, corresponding to $H_G^u$, are the vertices $b_2^v$ and $f_3^v$ and their incident edges.

\begin{figure}[!ht]
    \centering
    \begin{subfigure}{.45 \textwidth}
        \centering
        \includegraphics[width = \textwidth]{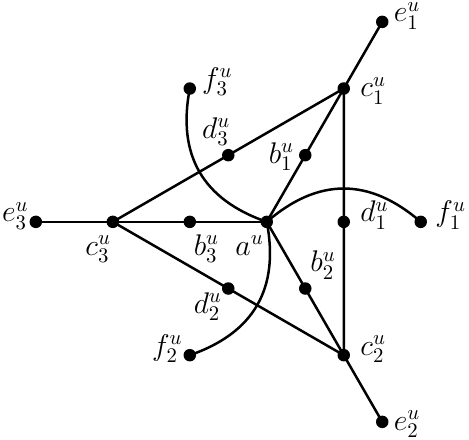}
        \caption{vertex-gadget for a degree 3 vertex.} \label{fig:vertex_gadget_3}
    \end{subfigure}
    \hfil
    \begin{subfigure}{.45 \textwidth}
        \centering
        \includegraphics[width = \textwidth]{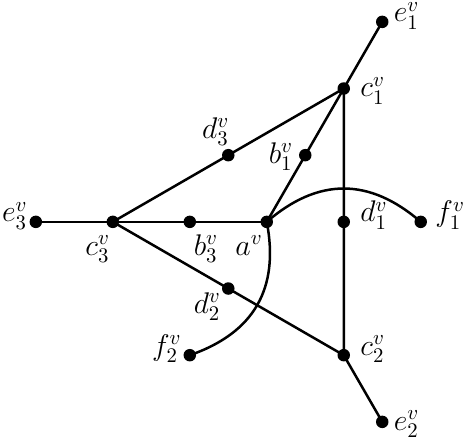}
        \caption{vertex-gadget for a degree 2 vertex.} \label{fig:vertex_gadget_2}
    \end{subfigure}
    \caption{The vertex-gadgets.} \label{fig:vertex_gadgets}
\end{figure}

Let $u$ be a degree 3 vertex in $G$ and let $v \in V(G)$ be such that $uv \in E(G)$. According to our construction of $G$, $v$ is always of degree 2. We connect the vertices of $H_G^u$ and $H_G^v$ using two vertices ($g_1^{uv}$ and $g_2^{uv}$) and five edges as shown in Figure \ref{fig:mini_edge_gadget}. Observe that $H_G^u$ and $H_G^v$ are connected to each other via vertices $f_1^u$ and $f_2^v$, thus $g_1^{uv}$ and $g_2^{uv}$ are connected to $f_1^u$'s and $f_2^v$'s associated sides. Since $u$ is of degree 3 in $G$, there are also similar connections between $H_G^u$ and two additional vertex-gadgets which are shown in the figure. One of them is labelled $H_G^t$, the other is not labelled. 

\begin{figure}[!ht]
    \centering
    \includegraphics[width =.6 \textwidth]{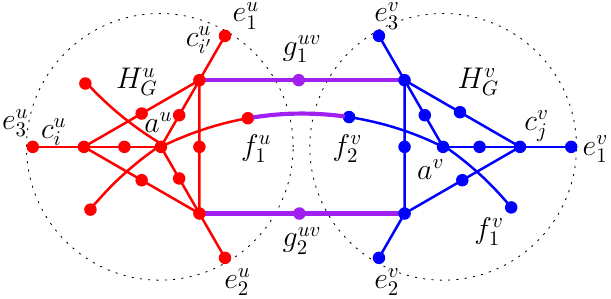}
    \caption{Connections between vertex-gadgets $H^u_G$ and $H^v_G$ (shown in the dashed circles), where $u$ is a degree~3 vertex and $v$ is a degree~2 vertex in $G$.}
    \label{fig:mini_edge_gadget}
\end{figure}

For a degree 2 vertex $v$ of $G$, the connections are similar with the only exception being that, instead of three, the vertex-gadget is connected to only two other vertex-gadgets. This means that $c_1^v$ and $c_3^v$ are both adjacent to one vertex of type $g_i^{xv}, i \in \{1,2\}, x \in V(G),$ while $c_2^v$, similar to the $c_j^u$'s ($u$ is of degree 3 in $G$), is adjacent to two such vertices.

Next, consider any two vertices $u, w \in V(G)$ such that the distance between them in $G$ is 2. Let us again assume $u$ to be of degree 3 in $G$. By the structure of $G$, this implies $w$ to be of degree~2. Let $v \in V(G)$ be the unique vertex adjacent to both $u$ and $w$. Let $f_{i'}^u$ (resp. $f_{j'}^w$) be the vertex in $H_G^u$ (resp. $H_G^w$) which is connected to $H_G^v$. Let $c_{i'}^u$ be the associate vertex of $f_{i'}^u$ and $c_h^w$ the associate vertex of $f_{j'}^w$. We connect $c_{i'}^u$ and $c_h^w$. This ensures that the length of the (unique) shortest path between $a^u$ and $a^w$ in $H_G$ is 5 and it passes through the edge $c_{i'}^u c_h^w$.

Now consider two vertices $t, w \in V(G)$ such that the distance between them in $G$ is 3. Let $u, v \in V(G)$ be the two vertices included in the unique shortest path between $t$ and $w$ in $G$. Also let $u$ be adjacent to $t$ and $v$ to $w$ in $G$. Let $f_{i'}^t$ (resp. $f_{j'}^w$) be the vertex in $H_G^t$ (resp. $H_G^w$) which is connected to $H_G^u$ (resp. $H_G^v$). Let $c_g^t$ and $c_h^w$ be the associate vertices of $f_{i'}^t$ and $f_{j'}^w$, respectively. We connect $c_g^t$ and $c_h^w$, such that the length of the (unique) shortest path between $a^t$ and $a^w$ is 5 and it passes through the edge $c_g^t c_h^w$. See~Figure~\ref{fig:edge_gadget} for reference.

\paragraph{Description of the various gadgets and connections.} In Figure~\ref{fig:edge_gadget}, the vertex-gadget colored \textcolor{red}{\textbf{red ($u$)}} is corresponding to a degree $3$ vertex. The vertex-gadgets adjacent to $u$ (which are all of degree $2$) are colored in \textcolor{ForestGreen}{\textbf{forest-green}},  \textcolor{blue}{\textbf{blue ($v$)}} and \textcolor{pink}{\textbf{pink ($t$)}}, respectively. The edges connecting $H^u_G$ and $H^v_G$ are shown in \textcolor{Plum}{\textbf{violet}}. The edges connecting $H^u_G$ and $H^t_G$ are shown in \textcolor{purple}{\textbf{purple}}. The vertex-gadget colored \textcolor{Goldenrod}{\textbf{yellow ($w$)}} is at a distance $2$ from $u$ and adjacent to $v$. The edges shown in \textcolor{green}{\textbf{green}} connect the gadgets of $v$ and $w$. The edges connecting $u$ and $w$ are highlighted in \textcolor{orange}{\textbf{orange}}. All such other connections have been highlighted in different colors for ease of understanding. The edges of the vertex-gadgets are thinner than the edges connecting the gadgets.

\begin{figure}[!ht]
    \centering
    \includegraphics[width =.75 \textwidth]{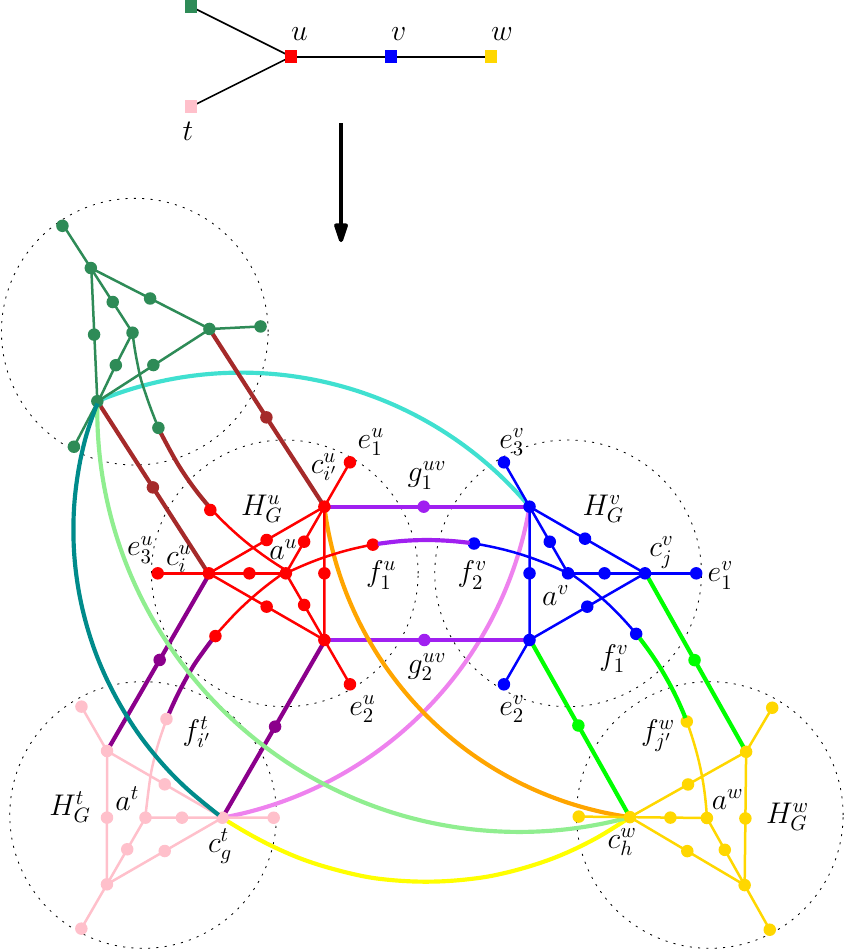}
    \caption{A sample of the original graph $G$ with its vertices (shown as squares), along with the constructed graph $H_G$. Here, vertex $u$ is of degree~3, while the other vertices are of degree~2. The dashed circles contain the vertex-gadgets. Connections between vertex-gadgets in $H_G$ are highlighted with thicker lines.}
    \label{fig:edge_gadget}
\end{figure}

The subgraphs $H_G^u, \forall u \in V(G)$, along with the connections between them constitute the graph $H_G$. Recall that the maximum degree of a vertex in $G$ is 3 and any two degree 3 vertices of $G$ are at least distance 5 apart. Therefore, for any vertex $u \in V(G)$, the maximum number of vertices at distance either exactly 2 or 3 in $G$ is 3. If $u$ is of degree 3 in $G$ then any vertex $c_i^u, i \in \{1,2,3\}$, is connected to four vertices from $V(H_G^u)$, two vertices of type $g_j^{ux}, j \in \{1,2\}, x \in V(G)$, and one vertex each of type $c_h^y$, where $h \in \{1, 3\}$ and $y \in V(G)$ is either at a distance 2 or 3 from $u$. This makes the total degree of $c_i^u$ to be 8. If $v$ is of degree 2 in $G$ then a vertex $c_i^v, i \in \{1,2,3\}$, is connected to at most 4 vertices from $V(H_G^v)$, at most two vertices of type $g_j^{vx}, j \in \{1,2\}, x \in V(G)$, and at most 2 vertex each of type $c_{h'}^{y'}$, where $h' \in \{1,2,3\}$ and $y' \in V(G)$ is either at a distance 2 or 3 from $v$. This makes the total degree of $c_i^v$ to be at most 10. A close inspection of $H_G$ reveals that this upper bound is actually 9 and the maximum degree of $H_G$ is determined by some $c_i^v$ and hence is also 9. We prove the following result.

\paragraph{Intuition of the reduction.} Before delving into the technical proof of validity of the reduction in the next lemma, let us first provide some intuition behind the design of the vertex-gadgets. We begin by focusing on the three pendant vertices in each vertex-gadget, which must belong to any MEG-set by Lemma~\ref{leaf-node}. Any two pendant vertices of the same vertex-gadget, say $e^u_i$ and $e^u_j$, will monitor all the edges along the unique shortest path between them. As a result, all the boundary edges of a vertex-gadget are effectively monitored. Moreover, these pendant vertices also monitor the edges used to connect two adjacent vertex-gadgets e.g. $H^u_G$ and $H^v_G$, specifically those involving the vertices $g^{uv}_1$ and $g^{uv}_2$. (Refer to Figure~\ref{fig:mini_edge_gadget} or Figure~\ref{fig:edge_gadget}.)

Now, consider including the central vertex $a^u$ in the MEG-set. This vertex monitors all paths from $a^u$ to each of the pendant vertices $e^u_i$ of the same vertex-gadget, thereby covering the internal edges of the gadget. Moreover, it also helps monitoring the edges involving the vertices of type $f^u_i$, and the internal edges of adjacent vertex-gadgets, such as the edges of the path of length~3 from $a^v$ to $e_1^v$ in Figure~\ref{fig:mini_edge_gadget} and Figure~\ref{fig:edge_gadget}.

On the other hand, in order to monitor the edges of type $f_i^uf_j^v$, one can show that one necessarily needs to include a vertex from one of the vertex-gadgets $H^u_G$ and $H^v_G$ in the MEG-set. 
%The only remaining edges to be monitored are those involving the vertices of type $f^u_i$. For these, at least one of the central vertices from two adjacent vertex-gadgets --— either $a^u$ or $a^v$ or both --— must be included in the MEG-set to ensure coverage of the edges passing through $f^u_i$. 
This design mirrors the behavior of a vertex cover in $G$, where each edge must have at least one of its endpoints included in the cover.

\begin{lemma} \label{lem:meg}
    $G$ has a vertex cover of size $k$ if and only if $H_G$ has a MEG-set of size $3n + k$.
\end{lemma}

\begin{proof}
    \textbf{(The if part.)} Let $S \subseteq V(H_G)$ be an MEG-set of $H_G$ of size $3n + k$ and let $S' \subset S$ be the set which remains after removing all degree 1 vertices from $S$. Since $H_G$ has $3n$ degree 1 vertices, all of which have to be included in every MEG-set, the size of $S'$ is $k$. Let $C \subseteq V(G)$ be the set of all vertices $v \in V(G)$ for which $V(H_G^v)$ has at least one element in $S$. It is easy to see that the size of $C$ is at most $k$. We show that it is also a vertex cover of $G$.

    Let us assume for the sake of contradiction that $C$ is not a vertex cover of $G$. Then, there exists an edge $uv \in E(G), u, v \in V(G)$, such that neither $u$ nor $v$ belong to $C$. Let us assume $u$ to be a degree 3 vertex of $G$ and let $H_G^u$ and $H_G^v$ be connected as shown in Figure \ref{fig:edge_gadget}. The other case, when $u$ is of degree 2, is similar. Now, consider the edge $f_1^u f_2^v \in E(H_G)$. Since $S$ is still an MEG-set, there exists a pair of vertices $x, y \in S$ which monitors the edge $f_1^u f_2^v$ i.e. every shortest path from $x$ to $y$ passes through $f_1^u f_2^v$. Let $\pi_{xy}$ be one such path. Without loss of generality we assume that $f_1^u$ is encountered before $f_2^v$ while traversing $\pi_{xy}$ from $x$ to $y$.

    Let $V'(H_G^v) = V(H_G^v) \setminus \{e_1^v, e_2^v, e_3^v\}$. Let $x'$ (resp. $y'$) be the first vertex from $V'(H_G^u)$ (resp. $V'(H_G^v)$) encountered when traversing $\pi_{xy}$ starting at $x$ (resp. $y$). Since both $u$ and $v$ do not belong to $C$, from our contruction of $C$, no element of $S$ belongs to $V'(H_G^u) \cup V'(H_G^v)$. This implies that neither $x$ nor $y$ belong to $V'(H_G^u) \cup V'(H_G^v)$. Since $x$ does not belong to $V'(H_G^u)$, $x'$ is either $c_i^u, i \in \{1,2,3\}$, or $f_{i'}^u, i' \in \{2, 3\}$, and similarly since $y$ does not belong to $V'(H_G^v)$, $y'$ is either $c_j^v, j \in \{1,2,3\}$, or $f_1^v$. We have the following cases.
   
    \smallskip \noindent \textbf{Case 1 ($x'$ is $c_i^u$ and $y'$ is $c_j^v$).} Note that in this case $i, j \in \{1,2,3\}$. It is easy to see that any shortest path between $c_i^u$ and $c_j^v$ passes through either the vertex $g_1^{uv}$ or $g_2^{uv}$ and does not include the edge $f_1^v f_2^v$. This means that $x$ and $y$ cannot monitor the edge $f_1^u f_2^v$, which implies a contradiction.
    
    \smallskip \noindent \textbf{Case 2 ($x'$ is $c_i^u$ and $y'$ is $f_1^v$).} Note that in this case $i \in \{1,2,3\}$. Let $f_{j'}^w, j' \in \{1,2\}$ and $w \in V(G)$, be the vertex encountered just before $f_1^v$ when traversing $\pi_{xy}$ starting at $y$. Since the distance between $u$ and $w$ in $G$ is 2, by construction there exists an edge $c_{i'}^u c_h^w, i' = 1$ and $h \in \{1, 3\}$, in $G$. Therefore, the length of the shortest path between $c_i^u$ and $f_{j'}^w$ is at most 6 and passes through the edge $c_{i'}^u c_h^w$ and not $f_1^u f_2^v$. This means that $x$ and $y$ cannot monitor the edge $f_1^u f_2^v$, which is a contradiction. The other case when $x'$ is some $f_{g}^u, g \in \{2, 3\}$, and $y'$ some $c_{h}^v, h \in \{1,2,3\}$, can be similarly disproved.

    \smallskip \noindent \textbf{Case 3 ($x'$ is $f_i^u$ and $y'$ is $f_1^v$).} Note that in this case $i \in \{2, 3\}$. Let $f_{i'}^t, i' \in \{1,2\}$ and $t \in V(G)$, be the vertex encountered just before $f_i^u$ when traversing $\pi_{xy}$ starting at $x$. Similarly, let $f_{j'}^w, j' \in \{1,2\}$ and $w \in V(G)$, be the vertex encountered just before $f_1^v$ when traversing $\pi_{xy}$ starting at $y$. Since $u$ is of degree 3, both $t$ and $w$ are degree 2 vertices of $G$. Since the distance between $t$ and $w$ in $G$ is 3, by construction there exists an edge $c_g^t c_h^w, g, h \in \{1, 3\}$, in $H_G$. Therefore, there are at least two shortest path (of length 7) between $f_{i'}^t$ and $f_{j'}^w$, one which includes the edge $f_1^u f_2^v$ and the other which includes the edge $c_{g'}^t c_h^w$. This means that $x$ and $y$ cannot monitor the edge $f_1^u f_2^v$, a contradiction.

    \smallskip \noindent \textbf{(The only if part.)} Let $C \subseteq V(G)$ be a vertex cover of $G$ of size $k$. We construct the set $S \subset V(H_G)$ as follows. For each vertex $v \in V(G)$, insert into $S$ the vertices $e_1^v, e_2^v$ and $e_3^v$ of $H_G$. If $v \in C$ then we also add the vertex $a^v \in V(H_G^v)$ to $S$. Observe that the size of $S$ is $3n +k$. We show that it is also an MEG-set of $H_G$.

    For every vertex $u \in V(G)$, the vertices $e_1^u, e_2^u, e_3^u \in S$ monitor all the edges in the cycle $c_1^u, d_1^u, c_2^u, d_2^u, c_3^u$, $d_3^u$ of $H_G^u$, as well as the edges $c_i^v e_i^v, \forall i \in \{1,2,3\}$. This is because for each pair of vertices from $e_1^v, e_2^v$ and $e_3^v$, the shortest path between them is unique and hence all edges on the shortest path are monitored. For example, for the vertex pair $e_1^v$ and $e_2^v$, the unique shortest path is $e_1^v, c_1^v, d_1^v, c_2^v, e_2^v$ and all its edges are monitored by the vertex pair. Here we assume $u$ to be a degree 3 vertex of $G$; the other case, when $u$ is degree 2, is similar.
    
    Next, with respect to an edge $uv \in E(G)$, let the subgraphs $H_G^u$ and $H_G^v$ be connected as shown in Figure \ref{fig:edge_gadget}. Again, the shortest path between $e_1^u$ and $e_3^v$ (which passes through the vertex $g_1^{uv}$), and between $e_2^u$ and $e_2^v$ (which passes through the vertex $g_2^{uv}$) are unique. Therefore, all edges of $H_G$ incident on $g_1^{uv}$ and $g_2^{uv}$ are monitored.

    If $u \in V(G)$ is part of the vertex cover $C$, then $a^u \in S$ and $a^u, b_i^u, c_i^u, e_i^u$ is the unique shortest path between $a^u$ and $e_i^u$, $\forall i \in \{1,2,3\}$. Therefore, all edges on these paths are monitored.

    If $u \in V(G)$ is not part of the vertex cover $C$, then all three neighbors of $u$ in $G$ belong to $C$. Let $v \in V(G)$ be one such neighbor. Then $a^v \in V(H_G^v)$ is part of $S$. Observe that the only shortest path between $e_3^u$ and $a^v$ is the path $e_3^u, c_3^u, b_3^u, a^u, f_1^u, f_2^v, a^v$ of length 6. Therefore, all edges on this path are monitored. Any other path between $e_3^u$ and $a^v$ passes through either $c_1^u$ or $c_2^u$ and is of length at least 7. This is ensured in the construction phase by restricting $c_3^u$ from being directly connected to any $c_{i'}^w$, $i' \in \{1, 3\}$. Similarly, we can show that edges $a^u b_i^u$ and $b_i^u c_i^u, \forall i \in \{1,2\}$, are also monitored.
    
    Now consider all edges $c_i^t c_j^w, i, j \in \{1,2,3\}$ and $t, w \in V(G)$, of $H_G$ such that $t$ and $w$ are at distance 2 or 3 in $G$. The shortest path between $e_i^t$ and $e_j^w$ is of length 3, is unique, and passes through the edge $c_i^t c_j^w$, thereby monitoring it. Since we have shown all edges of $H_G$ to be monitored by some vertex pair in $S$, therefore $S$ is an MEG-set of $H_G$.
\end{proof}

Since \textsc{Monitoring Edge Geodetic Set} is clearly in NP~\cite{H23} and $H_G$ has maximum degree 9, Lemma~\ref{lem:vertex_cover} and Lemma~\ref{lem:meg} imply the following theorem.

\begin{theorem}
    \textsc{Monitoring Edge Geodetic Set} is NP-complete, even for graphs of maximum degree~9.
\end{theorem}

\section{Conclusion}\label{sec:conclu}

Inspired by a network monitoring application, we have defined the new concept of MEG-sets of a graph, which is a common refinement of the popular concept of a geodetic set and its variants, and of the previously studied distance-edge-monitoring sets.

We have studied the concept on basic graph classes. It is interesting to note that there are many graph classes which require the entire vertex set in any MEG-set: complete graphs, complete multipartite graphs, and hypercubes. More examples are provided in~\cite{H23}. It would be an interesting question to characterize all such graphs, if they can be described in a meaningful way.

Our upper bound using the feedback edge set number is probably not tight. What is a tight bound on this regard?

Finally, it remains to investigate further computational aspects of the problem. Clearly, \textsc{Monitoring Edge Geodetic Set} is polynomial-time solvable on the graph classes studied in Section~\ref{sec:classes}, such as trees, unicyclic graphs, etc. What about graphs of bounded tree-width? Also, \textsc{Vertex Cover} remains NP-hard on planar cubic graphs~\cite{mohar2001face}. However, our NP-hardness reduction does not preserve planarity. Is \textsc{Monitoring Edge Geodetic Set} NP-complete for planar graphs? For subcubic graphs? What about other standard graph classes like interval graphs? Also, the approximation complexity and the parameterized complexity of the problem could be investigated. Regarding parameterized complexity, parameters of interest are the solution size or structural parameters, like the feedback edge set number.

\subsubsection*{Acknowledgements} Sanjana Dey was partially supported by the Fonds de la Recherche Scientifique – FNRS under Grant n° T.0188.23 (PDR ControlleRS). Florent Foucaud was partially supported by the ANR project GRALMECO (ANR-21-CE48-0004), the French government IDEX-ISITE initiative 16-IDEX-0001 (CAP 20-25), the International Research Center ``Innovation Transportation and Production Systems'' of the I-SITE CAP 20-25, and the CNRS IRL ReLaX. Florent Foucaud thanks Ralf Klasing and Tomasz Radzik for initial discussions in Bordeaux in 2019, which inspired the present study. We thank Davide Bil\`o for pointing out a mistake in the initial statement of Corollary~\ref{cor:MLN}, and the anonymous referees for their careful reading and helpful comments.

\end{document}